\documentclass[11pt]{amsart}
\usepackage[T1]{fontenc}
\usepackage{lmodern,amssymb,mathtools,microtype,etoolbox}
\usepackage[colorlinks=true,linkcolor=blue,citecolor=blue,urlcolor=blue]{hyperref}
\makeatletter
\patchcmd{\section}{\scshape}{\bfseries}{}{}
\makeatother
\allowdisplaybreaks
\newtheorem{theorem}{Theorem}[section]
\newtheorem{proposition}[theorem]{Proposition}
\newtheorem{lemma}[theorem]{Lemma}
\newtheorem{corollary}[theorem]{Corollary}
\theoremstyle{remark}
\newtheorem{remark}[theorem]{Remark}
\newcommand{\CC}{\mathbb C}
\newcommand{\RR}{\mathbb R}
\newcommand{\ZZ}{\mathbb Z}

\title[Norm extraction and cubic descent]{Primitive norm extraction and a signed descent for the cubic Fermat equation}
\author{R. Laniewski}
\date{September 23, 2026}
\subjclass[2020]{Primary 11D41; Secondary 11D25, 11E25}
\keywords{Cubic Fermat equation, infinite descent, quadratic norm, primitive power extraction, Thue's lemma}

\begin{document}

\begin{abstract}
We give a self-contained proof of the cubic case of Fermat's theorem through a signed form of Euler's descent, with the required extraction step proved directly in $\mathbb{Z}[\sqrt{-3}]$. A pigeonhole argument of Thue represents the relevant primes by $X^2+3Y^2$, and explicit division in $\mathbb{Z}[\sqrt{-3}]$ lifts primitive odd-power norms to powers in that order. For cubes, the resulting coordinate identities transform $2K(K^2+3M^2)=C^3$ into another equation of the same signed form with a strictly smaller non-zero integer height. The construction gives $|C'|\le |C|/39$. Half-sums and half-differences of the two odd entries place the sum and difference parity configurations in this single descent.
\end{abstract}

\maketitle

\section{Introduction}\label{sec:introduction}
The equation $A^3+B^3=C^3$ has no solution in positive integers. Euler's treatment in the second volume of the 1770 \emph{Vollst\"andige Anleitung zur Algebra} brings the quadratic form $X^2+3Y^2$ into the descent \cite[Part~II, Chapter~15, \S243]{euler}. The step requiring further justification is the extraction of a cube in $\ZZ[\sqrt{-3}]$ from a primitive cube norm. Edwards discusses this gap and its completion in Chapter~2 of~\cite{edwards}. Dickson's second volume places the cubic case within the historical development of Fermat's equation \cite{dickson}. Barbara gives a short signed descent that treats both parity configurations together \cite{barbara}. The extraction step of that descent rests on unique factorization in the Eisenstein integers, together with the classification of their prime elements and the representation of primes $p\equiv1\pmod6$ by $X^2+3Y^2$. Section~\ref{sec:fermat} adapts this descent, while Section~\ref{sec:arithmetic-lifting} replaces its extraction step by an argument within $\ZZ[\sqrt{-3}]$.

The arithmetic takes place in the quadratic order $\mathcal L=\ZZ[\sqrt{-3}]$. A perfect-power norm does not by itself ensure a power in this order, and Remark~\ref{rem:units} exhibits the failure of unique factorization.

Theorem~\ref{thm:odd-norm-lifting} supplies the extraction statement for every odd exponent $k\geq3$. Its proof remains in $\mathcal L$ and requires neither passage to the maximal order nor the composition or genus theory of binary quadratic forms. A pigeonhole argument of Thue represents the relevant primes by $X^2+3Y^2$, and rational-prime divisibility and norm multiplicativity then permit explicit division by elements of prime norm. Each induction step extracts exactly $k$ copies of one prime-norm factor, including when the rational prime occurs repeatedly in the original norm. This gives an elementary route through the step left unjustified in the classical treatment.

For two odd integers $a,b$, both $a+b$ and $b-a$ are even. Thus $K=(a+b)/2$ and $M=(b-a)/2$ are integers. Applying these coordinates to the two odd entries of a primitive cubic solution preserves integrality in either parity configuration. The sum configuration gives $2K(K^2+3M^2)=C^3$, and the difference configuration gives the same expression after interchanging $K$ and $M$. Section~\ref{sec:fermat} proves that this signed equation has no primitive solution with non-zero coordinates and opposite parities. The descent preserves those conditions and reduces the height $|C|$ by at least a factor of $39$. Section~\ref{sec:parities} completes the passage from the two parity configurations to the full exponent-three theorem.

\subsection{Known approaches to the cubic case}\label{subsec:known-approaches}
The proofs of the cubic case fall broadly into two families. The first family passes to a richer structure, such as an order with unique factorization, the composition of binary quadratic forms, or the arithmetic of an elliptic curve. The second family remains among the rational integers and replaces this structure by explicit divisibility lemmas.

Within the first family, the proof most often found in textbooks takes place in the ring $\ZZ[\omega]$ of Eisenstein integers, where $\omega=(-1+\eta)/2$ and $\eta=\sqrt{-3}$. This ring is the maximal order of $\mathbb Q(\eta)$ and a Euclidean domain, so unique factorization settles the extraction of cubes at once. The argument goes back to Gauss and appears, for instance, in \cite[Chapter~17]{irelandrosen}. This route introduces half-integral coordinates and the six units $\pm1,\pm\omega,\pm\omega^2$, which the descent must then absorb. Monsky gives a self-contained account of descent in $\ZZ[\omega]$ written for a broad audience \cite{monsky}. Barbara's signed descent also belongs to this route. Its extraction lemma relies on unique factorization in $\ZZ[\omega]$, although its prime elements are chosen inside $\ZZ[\sqrt{-3}]$ and the units are then removed by parity \cite{barbara}.

Another route in the first family stays with the form $X^2+3Y^2$ and derives the extraction from the classical theory of binary quadratic forms. This route rests on three ingredients. The reduction theory of Lagrange and Gauss shows that every primitive positive definite form of discriminant $-12$ is properly equivalent to $X^2+3Y^2$, so the class number $h(-12)$ equals one \cite[\S2]{cox}. Gauss's composition of forms, introduced in Section~V of the \emph{Disquisitiones Arithmeticae}, makes the classes of primitive forms of a given discriminant into a finite abelian group and relates the representations of a product to those of its factors \cite[Section~V]{gauss}, \cite[\S3]{cox}, \cite{buell}. The discriminant $-12$ equals $2^2$ times the discriminant $-3$ of the maximal order, so it is not fundamental. Its forms correspond to the proper ideals of the order $\ZZ[\sqrt{-3}]$ of conductor two, and this correspondence is compatible with composition \cite[\S7]{cox}.

With these results, a primitive representation $X^2+3Y^2=W^3$ with $W$ odd first gives $\gcd(W,6)=1$. If $3\mid W$, then primitivity would force the exponent of $3$ in $X^2+3Y^2$ to be exactly one, which is incompatible with a cube. The principal ideal generated by $X+Y\sqrt{-3}$ has norm $W^3$ and is prime to the conductor. Its prime ideal factors lie above split primes, and primitivity excludes the simultaneous occurrence of a prime ideal and its conjugate. Unique factorization for ideals prime to the conductor therefore makes every exponent a multiple of three \cite[\S7]{cox}. The ideal is a cube, the triviality of the class group makes its cube root principal, and the units $\pm1$ absorb the remaining sign. This route is rigorous, although reduction, composition, and the correspondence with ideals of a non-maximal order together demand considerably more theory than the descent itself.

A third route in the first family passes to an elliptic curve. The smooth projective Fermat cubic is isomorphic over $\mathbb Q$ to the elliptic curve with affine equation $y^2=x^3-432$ \cite{silvermantate}. The exponent-three theorem is equivalent to the assertion that its rational points are precisely the point at infinity and $(12,\pm36)$. Descent on this curve gives a proof through its group law.

In the second family, Edwards completes Euler's argument with tools close to those available to Euler \cite[Chapter~2]{edwards}. The completion proceeds through a sequence of divisibility lemmas for the form $X^2+3Y^2$, among them the representation of suitable primes and the division of one representation by another. Weil's historical account situates Euler's argument within the broader study of this form \cite{weil}. These approaches are elementary, although the number of auxiliary lemmas makes them lengthy.

The present paper belongs to the second family and uses the order $\mathcal L=\ZZ[\eta]$ as a language for the division steps. Lemma~\ref{lem:fermat-prime} supplies elements of prime norm, and the explicit division~\eqref{eq:fermat-division} permits successive extraction of their powers. Under the primitive odd-norm hypotheses, Theorem~\ref{thm:odd-norm-lifting} thereby gives the required power extraction within $\mathcal L$. The coordinates remain integral throughout, and the only units are $\pm1$.

\section{Norms and primitive power extraction}\label{sec:arithmetic-lifting}
Let $\eta=\sqrt{-3}$ and let $\mathcal L=\ZZ[\eta]\subset\CC$. Conjugation sends $X+Y\eta$ to $X-Y\eta$. We call an element primitive when its two integer coordinates are coprime. The multiplication law and norm are
\begin{equation}\label{eq:norm-product}
\begin{aligned}
(X+Y\eta)(E+F\eta)&=(XE-3YF)+(XF+YE)\eta\\
\mathcal N(X+Y\eta)&=(X+Y\eta)(X-Y\eta)=X^2+3Y^2.
\end{aligned}
\end{equation}
The norm is a non-negative integer, is positive away from zero, and is multiplicative. Opposite parity of $X,Y$ is equivalent to oddness of this norm. These elementary properties will suffice for the extraction argument.

\subsection{Prime norms from a pigeonhole argument}\label{subsec:fermat-prime-norms}
\begin{lemma}[A prime represented by the norm]\label{lem:fermat-prime}
Let $p>3$ be a prime. Suppose that $s^2\equiv-3\pmod p$ has an integer solution. Then there are integers $E,F$ with
\[
p=E^2+3F^2.
\]
They are coprime and of opposite parities, and $p\nmid EF$.
\end{lemma}
\begin{proof}
We use the pigeonhole argument of Thue \cite{thue}. Put $m=\lfloor\sqrt p\rfloor$. Since $p$ is prime, it is not a square, so $m<\sqrt p$. The $(m+1)^2$ pairs $(x,y)$ with $0\leq x,y\leq m$ outnumber the $p$ residues modulo $p$. Two distinct pairs therefore give the same residue of $x-sy$. Their difference $(x,y)$ is non-zero and satisfies
\[
x\equiv sy\pmod p,
\qquad |x|<\sqrt p,
\qquad |y|<\sqrt p.
\]
If $y=0$, then $p\mid x$ with $|x|<p$, which forces $x=0$. Hence $y\neq0$. The congruence gives $x^2+3y^2\equiv(s^2+3)y^2\equiv0\pmod p$, while
\[
0<x^2+3y^2<4p.
\]
The integer $x^2+3y^2$ is therefore $p$, $2p$ or $3p$. Squares modulo $4$ show that $x^2+3y^2$ never equals $2\pmod4$, excluding $2p$. If $x^2+3y^2=3p$, then $3\mid x$ and $p=y^2+3(x/3)^2$. In the remaining case $p=x^2+3y^2$. Either way, $p=E^2+3F^2$ for integers $E,F$.

Primality gives coprimality. Since $p$ is odd, $E,F$ have opposite parities. The identities $E=0$ and $F=0$ would make $p$ three times a square or a square. The bounds $|E|<p$ and $|F|<p$ now give $p\nmid EF$.
\end{proof}

\subsection{Primitive odd-power extraction}\label{subsec:odd-norm-lifting}
\begin{theorem}[Lifting a primitive odd-power norm]\label{thm:odd-norm-lifting}
Let $k\geq3$ be odd, and let $X,Y\in\ZZ$ be coprime and of opposite parities. Suppose $X^2+3Y^2=W^k$ with $W>0$ an integer. Then there are coprime integers $E,F$ of opposite parities such that
\begin{equation}\label{eq:odd-norm-lifting}
X+Y\eta=(E+F\eta)^k,
\qquad E^2+3F^2=W.
\end{equation}
\end{theorem}
\begin{proof}
The norm is odd, so $W$ is odd. Also $3\nmid W$. Indeed, if $3\mid W$, then $3\mid X$ and $3\nmid Y$ by coprimality. The exponent of $3$ in $X^2+3Y^2$ would then be exactly one, which is impossible for a $k$th power with $k\geq3$.

Thus $W$ is odd and $3\nmid W$, so every prime divisor of $W$ is greater than three. We prove~\eqref{eq:odd-norm-lifting} by induction on $W$ within this range. For $W=1$, one has $Y=0$ and $X=\pm1$, giving the assertion because $k$ is odd. For $W>1$, choose a prime $p\mid W$. Coprimality gives $p\nmid Y$, and $(XY^{-1})^2\equiv-3\pmod p$. Lemma~\ref{lem:fermat-prime} provides $p=e^2+3f^2$. Write $\pi=e+f\eta$ and $\overline\pi=e-f\eta$.

We first check division explicitly. Let $A+B\eta$ be primitive with norm divisible by $p$. Then $p\nmid B$, and $AB^{-1}$ is a root of $z^2=-3$ in the field $\ZZ/p\ZZ$. The two roots are $\pm ef^{-1}$, and they are distinct, since $p$ is odd and $p\nmid ef$. Hence exactly one of the congruences $Af\equiv Be$ and $Af\equiv-Be$ holds modulo $p$. When $Af\equiv Be\pmod p$, one has
\begin{equation}\label{eq:fermat-division}
\frac{(A+B\eta)\overline\pi}{p}
=\frac{Ae+3Bf}{p}+\frac{Be-Af}{p}\eta
\in\mathcal L.
\end{equation}
Both numerator divisibilities follow from $Af\equiv Be$ and $e^2+3f^2=p$. Conversely, the coefficient of $\eta$ in~\eqref{eq:fermat-division} shows that $\pi$ divides $A+B\eta$ only when $Af\equiv Be\pmod p$. The same computation with $f$ replaced by $-f$ applies to $\overline\pi$. Thus exactly one of the two conjugate factors divides $A+B\eta$ within $\mathcal L$.

Apply this division to $\alpha=X+Y\eta$. After a possible change of sign of $f$, the factor $\pi$ divides $\alpha$, and we write $\alpha=\pi\beta_1$. From this point on, $\pi$ remains fixed for the remaining $k-1$ divisions. Each quotient remains primitive, because a common integer divisor of its two coordinates would also divide both coordinates of $\alpha$. Its norm remains odd. We perform exactly $k$ divisions. The exponent of $p$ in the initial norm $W^k$ is at least $k$, so before each of these divisions the quotient norm is divisible by $p$. By the preceding paragraph, exactly one of $\pi$ and $\overline\pi$ divides each such quotient. Suppose that after $m$ divisions, with $1\leq m<k$, we have $\alpha=\pi^m\beta_m$. If the next division were by $\overline\pi$, so that $\beta_m=\overline\pi\delta$ with $\delta\in\mathcal L$, then
\[
\alpha=\pi^m\overline\pi\delta
=p\bigl(\pi^{m-1}\delta\bigr).
\]
Since $\pi^{m-1}\delta\in\mathcal L$, both integer coordinates $X,Y$ of $\alpha$ would be divisible by $p$, contradicting $\gcd(X,Y)=1$. Thus all $k$ divisions use $\pi$. We stop after the $k$th division, even when $p$ still divides the remaining norm, and obtain
\[
\alpha=\pi^k\beta,
\qquad \mathcal N(\beta)=(W/p)^k.
\]
The coordinates of $\beta$ are primitive and of opposite parities. Induction writes $\beta=\gamma^k$ with $\gamma\in\mathcal L$. Multiplication in $\mathcal L$ is commutative, so $\alpha=(\pi\gamma)^k$.

Write $\pi\gamma=E+F\eta$. Multiplicativity of the norm yields $W=E^2+3F^2$. A common divisor of $E,F$ would have its $k$th power divide $X,Y$, while oddness of $W$ gives opposite parities. This completes the induction and the proof.
\end{proof}

\subsection{Primitive cube extraction}\label{subsec:fermat-cube}
\begin{lemma}[Primitive cube norms]\label{lem:fermat-cube}
Let $X,Y\in\ZZ$ be coprime and of opposite parities. Suppose that
\begin{equation}\label{eq:fermat-cube-norm}
X^2+3Y^2=W^3,
\qquad W\in\ZZ_{>0}.
\end{equation}
Then there are coprime integers $E,F$ of opposite parities such that
\begin{equation}\label{eq:fermat-cube-extract}
X+Y\eta=(E+F\eta)^3,
\qquad W=E^2+3F^2.
\end{equation}
In particular,
\begin{equation}\label{eq:fermat-cube-coordinates}
X=E(E^2-9F^2),
\qquad Y=3F(E^2-F^2),
\qquad 3\mid Y.
\end{equation}
\end{lemma}
\begin{proof}
Apply Theorem~\ref{thm:odd-norm-lifting} with $k=3$ and expand $(E+F\eta)^3$ using $\eta^2=-3$.
\end{proof}

For instance, $10^2+3\cdot9^2=7^3$ and $10+9\eta=(-2+\eta)^3$. Primitivity remains necessary even when the norm is odd. Indeed,
\begin{equation}\label{eq:fermat-nonprimitive-example}
\mathcal N(14+7\eta)=14^2+3\cdot7^2=7^3.
\end{equation}
A cube root in $\mathcal L$ would have norm $7$, whose representatives are $\pm2\pm\eta$. Their cubes all have real coordinate $\pm10$, so none equals $14+7\eta$.

\begin{remark}[Units and the restricted extraction]\label{rem:units}
The units of $\mathcal L$ are exactly $\{1,-1\}$, since a unit has norm one and $E^2+3F^2=1$ forces $F=0$ and $E=\pm1$. Both units are $k$th powers for odd $k$.

The extraction theorem uses the primitive odd-norm hypotheses. Unique factorization fails in $\mathcal L$. Indeed, $4=2\cdot2=(1+\eta)(1-\eta)$ gives distinct irreducible factorizations, since norm $2$ is impossible and the units are $\pm1$. The prime-by-prime division in the proof establishes exactly the restricted extraction that is needed. The pigeonhole argument of Lemma~\ref{lem:fermat-prime} supplies the required prime representation by $X^2+3Y^2$.
\end{remark}

\begin{remark}[Uniqueness of the extracted root]\label{rem:uniqueness}
The element $E+F\eta$ in Theorem~\ref{thm:odd-norm-lifting} is unique. Indeed, put $\gamma=E+F\eta$. Every other $k$th root of $\gamma^k$ in $\CC$ has the form $\zeta\gamma$ with $\zeta^k=1$ and $\zeta\neq1$. Suppose that $\zeta\gamma\in\mathcal L$. Then $\zeta$ lies in $\mathbb Q(\eta)$, whose roots of unity are the sixth roots of unity. As $k$ is odd, the order of $\zeta$ divides $\gcd(6,k)=\gcd(3,k)$. If $3\nmid k$, then $\zeta=1$, contrary to the choice of $\zeta$. If $3\mid k$, then $\zeta$ is one of $\omega$ and $\overline\omega$, where $\omega=(-1+\eta)/2$. The identities
\[
\omega\gamma=\frac{-E-3F}{2}+\frac{E-F}{2}\eta,
\qquad
\overline\omega\gamma=\frac{-E+3F}{2}-\frac{E+F}{2}\eta
\]
show that $\zeta\gamma\in\mathcal L$ would require $E\equiv F\pmod2$, contrary to the opposite parities. For the example above, the two other cube roots of $10+9\eta$ are $(-1-3\eta)/2$ and $(5+\eta)/2$.

The uniqueness persists among quaternions with integer coordinates. Identify $\eta$ with $\mathbf i+\mathbf j+\mathbf k$ in the Lipschitz order, so that $\mathcal N$ becomes the quaternion norm. A quaternion $q$ with $q^k=\alpha$ commutes with $\alpha$. If $Y\neq0$, then $q$ lies in $\RR+\RR\eta$, whose intersection with the Lipschitz order is $\mathcal L$. If $Y=0$, then primitivity gives $\alpha=\pm1$, and multiplicativity of the quaternion norm gives $\mathcal N(q)=1$. The integer quaternions of norm one are $\pm1,\pm\mathbf i,\pm\mathbf j,\pm\mathbf k$. Since $k$ is odd, the non-real members have non-real $k$th powers, leaving only $q=\alpha$. The Hurwitz order, which also contains the quaternions with four half-integral coordinates, meets $\RR+\RR\eta$ in $\ZZ[\omega]$ and contains the cube root $\tfrac12(1+\mathbf i+\mathbf j+\mathbf k)$ of $-1$. The passage from $\mathcal L$ to $\ZZ[\omega]$ thus mirrors the passage from the Lipschitz order to the Hurwitz order.
\end{remark}

\section{A signed descent for the cubic equation}\label{sec:fermat}
The half-sum and half-difference of two odd integers produce integral coordinates in the norm $K^2+3M^2$. Primitive cube extraction then permits a descent whose signed formulation is preserved when the coordinates change order.

\subsection{The half-sum, the half-difference, and the norm}\label{subsec:fermat-reduction}
Suppose that $a,b,c$ are pairwise coprime positive integers satisfying
\begin{equation}\label{eq:fermat-start}
a^3+b^3=c^3,
\qquad a<b<c,
\qquad a,b\text{ odd},
\qquad c\text{ even}.
\end{equation}
Set
\begin{equation}\label{eq:fermat-km}
K=\frac{a+b}{2},
\qquad M=\frac{b-a}{2},
\qquad a=K-M,
\qquad b=K+M.
\end{equation}
Then $K>M>0$, $\gcd(K,M)=1$, and $K,M$ have opposite parities. Indeed, every common divisor of $K,M$ divides $a,b$, and $K+M=b$ is odd. Expansion gives the equivalent equation
\begin{equation}\label{eq:fermat-norm}
2K(K^2+3M^2)=c^3.
\end{equation}
Conversely, positive coprime integers $K>M$ of opposite parities satisfying~\eqref{eq:fermat-norm} yield odd coprime $a,b$ through~\eqref{eq:fermat-km}. The equation makes $c$ even and coprime to both $a$ and $b$, and $c>b$ follows from $c^3-b^3=a^3>0$.

The equation can be written as $2K\mathcal N(K+M\eta)=c^3$. Its factors have only the prime $3$ as a possible common divisor. The divisibility analysis below produces a primitive norm that is a cube, to which Lemma~\ref{lem:fermat-cube} applies.

\subsection{Divisibility and a decreasing integer height}\label{subsec:fermat-descent}
To keep the descent closed under changes of sign, consider the enlarged system
\begin{equation}\label{eq:fermat-signed}
2K(K^2+3M^2)=C^3,
\qquad KMC\neq0,
\qquad \gcd(K,M)=1,
\qquad K+M\text{ odd}.
\end{equation}
Here $K,M,C$ are signed integers. The positive sum configuration is included directly. The difference configuration will enter through the interchange of $K$ and $M$ in Subsection~\ref{subsec:fermat-parities}. Allowing signs removes the need to preserve the ordering $K>M>0$ during descent.

\begin{proposition}[Necessary divisibility]\label{prop:fermat-divisibility}
Every solution of~\eqref{eq:fermat-signed} would satisfy
\begin{equation}\label{eq:fermat-divisibility}
4\mid K,
\qquad M\text{ odd},
\qquad 9\mid K,
\qquad 3\nmid M.
\end{equation}
Moreover, there would be non-zero integers $u,v$, with $v>0$, such that
\begin{equation}\label{eq:fermat-split}
\frac{2K}{9}=u^3,
\qquad M^2+3(K/3)^2=v^3,
\qquad C=3uv.
\end{equation}
\end{proposition}
\begin{proof}
Write $D=K^2+3M^2$, which is odd. The cube $2KD$ is even, hence divisible by $8$, so $4\mid K$ and $M$ is odd. Also
\[
\gcd(2K,D)=\gcd(K,3).
\]
If $3\nmid K$, then the factors $2K,D$ are coprime and $D$ is a positive cube. Lemma~\ref{lem:fermat-cube} gives $3\mid M$. Consequently $C^3\equiv2K^3\equiv\pm2\pmod9$, contradicting the cube residues $0,1,-1$. Thus $3\mid K$ and $3\nmid M$. The exponent of $3$ in $D$ is exactly one. The cube equation now forces the exponent of $3$ in $K$ to be congruent to two modulo three, so $9\mid K$.

Dividing the equation by $27$ gives
\[
\frac{2K}{9}\,\frac{D}{3}=(C/3)^3.
\]
The two factors are coprime. The second is odd and prime to $3$, while every common prime other than $3$ would divide $K$ and $M$. Hence each factor is an integer cube. Since $D/3=M^2+3(K/3)^2>0$, the cube roots give~\eqref{eq:fermat-split}.
\end{proof}

\begin{theorem}[Descent for the cubic equation]\label{thm:fermat-descent}
System~\eqref{eq:fermat-signed} has no solution. In particular, the parity configuration in~\eqref{eq:fermat-start} is impossible.
\end{theorem}
\begin{proof}
Suppose there is a solution with $|C|$ minimal. Apply Proposition~\ref{prop:fermat-divisibility}. The integers $M,K/3$ are coprime and of opposite parities, so Lemma~\ref{lem:fermat-cube} gives
\begin{equation}\label{eq:fermat-ef}
M=E(E^2-9F^2),
\qquad K/3=3F(E^2-F^2),
\qquad v=E^2+3F^2.
\end{equation}
Here $\gcd(E,F)=1$. Since $M$ is odd, $E$ is odd and $F$ is even. The non-zero values of $K,M$ ensure $EF(E-F)(E+F)\neq0$. Equations~\eqref{eq:fermat-split} and~\eqref{eq:fermat-ef} give
\begin{equation}\label{eq:fermat-three-factors}
(2F)(E+F)(E-F)=u^3.
\end{equation}
These factors are pairwise coprime. An odd common prime would divide both $E$ and $F$, and the two latter factors are odd. Therefore there are non-zero integers $r,s,t$ with
\begin{equation}\label{eq:fermat-rst}
E+F=r^3,
\qquad E-F=s^3,
\qquad 2F=t^3,
\qquad u=rst.
\end{equation}
The integers $r,s$ are odd and coprime, and $r^3-s^3=t^3$. Define
\begin{equation}\label{eq:fermat-step}
K'=\frac{r-s}{2},
\qquad M'=\frac{r+s}{2},
\qquad C'=t.
\end{equation}
Neither $K'$ nor $M'$ vanishes, since these alternatives would give $F=0$ or $E=0$. They are coprime, and their sum is the odd integer $r$. Expansion now yields
\[
2K'\bigl((K')^2+3(M')^2\bigr)=t^3=(C')^3.
\]
Thus~\eqref{eq:fermat-step} is another solution of the same signed system. Since $E$ is odd, $|E|\geq1$. The integer $F$ is even and non-zero, because $F=0$ would imply $K=0$ in~\eqref{eq:fermat-ef}. Hence $|F|\geq2$ and $v=E^2+3F^2\geq1^2+3\cdot2^2=13$. Its height therefore satisfies
\begin{equation}\label{eq:fermat-height}
0<|C'|=|t|=\frac{|C|}{3|rsv|}\leq\frac{|C|}{39}<|C|.
\end{equation}
This contradicts minimality and proves the theorem.
\end{proof}

The closure of the argument can therefore be read in three steps. Primitive norm extraction supplies~\eqref{eq:fermat-ef}, coprime integer factors supply~\eqref{eq:fermat-rst}, and~\eqref{eq:fermat-height} supplies the strictly decreasing positive integer. The positive integer height governs the descent through all changes of sign and order.

\section{The two parity configurations}\label{sec:parities}
\subsection{The difference-of-cubes configuration}\label{subsec:fermat-parities}
The half-sum and half-difference remain integral in both parity configurations by applying them to the two odd entries. When the even entry is a summand, moving the odd summand to the other side gives a difference of cubes.

\begin{proposition}[The difference-of-cubes configuration]\label{prop:fermat-difference}
There are no pairwise coprime positive integers $a,b,c$ with $a,b$ odd and $c$ even satisfying $b^3-a^3=c^3$.
\end{proposition}
\begin{proof}
Suppose such integers exist. Necessarily $b>a$, and
\[
K=\frac{a+b}{2},
\qquad M=\frac{b-a}{2}
\]
are coprime positive integers of opposite parities with $K>M$. Expanding the difference gives
\begin{equation}\label{eq:fermat-difference}
(K+M)^3-(K-M)^3=2M(M^2+3K^2)=c^3.
\end{equation}
The triple $(M,K,c)$ therefore satisfies~\eqref{eq:fermat-signed}. Theorem~\ref{thm:fermat-descent} excludes it. The descent imposes no ordering on its two integer coordinates, so the inequality $M<K$ creates no additional case.
\end{proof}

Under the interchange of coordinates, Proposition~\ref{prop:fermat-divisibility} would require
\[
4\mid M,
\qquad 9\mid M,
\qquad K\text{ odd},
\qquad 3\nmid K.
\]
The corresponding norm is $\mathcal N(M+K\eta)=M^2+3K^2$. More explicitly, the preliminary cube extraction would give
\[
\frac{2M}{9}=u^3,
\qquad K^2+3(M/3)^2=v^3,
\qquad c=3uv.
\]
The same primitive cube-norm lemma and the same height decrease~\eqref{eq:fermat-height} then apply. Thus the sum and difference configurations enter a single signed descent while preserving integral half-coordinates.

\subsection{Completion of the exponent-three proof}\label{subsec:fermat-completion}
The two parity configurations have the respective reductions
\begin{equation}\label{eq:fermat-two-cases}
\begin{aligned}
a^3+b^3=c^3&\quad\Longrightarrow\quad 2K(K^2+3M^2)=c^3\\
b^3-a^3=c^3&\quad\Longrightarrow\quad 2M(M^2+3K^2)=c^3.
\end{aligned}
\end{equation}
In both lines, $a,b$ are the odd entries, $c$ is the even entry, and $K=(a+b)/2$, $M=(b-a)/2$ are integers. After primitive normalization and ordering the odd entries, these coordinates are positive, coprime, and of opposite parities. The first line enters the signed descent as $(K,M,c)$ and the second as $(M,K,c)$. The following corollary makes the exhaustion of the cases explicit.

\begin{corollary}[The full exponent-three case]\label{cor:fermat-three}
There are no positive integers $A,B,C$ satisfying $A^3+B^3=C^3$. Equivalently, there are no non-zero integers $x,y,z$ satisfying $x^3+y^3+z^3=0$.
\end{corollary}
\begin{proof}
Divide a putative positive solution by $\gcd(A,B,C)$. The resulting triple is pairwise coprime, because any prime dividing two entries must also divide the third. Reduction modulo $2$ shows that exactly one entry is even. All three even entries would contradict primitiveness, while one or three odd entries cannot satisfy the equation.

If $C$ is even, then $A,B$ are odd. Equality $A=B$ is impossible, since $2A^3=C^3$ would give incompatible exponents of $2$ modulo three. Interchanging the summands as needed gives $A<B<C$. This is the sum configuration~\eqref{eq:fermat-start}, excluded by Theorem~\ref{thm:fermat-descent}.

If $C$ is odd, then exactly one summand is even. Name the odd summand $a$, the even summand $c$, and put $b=C$. The equation becomes $b^3-a^3=c^3$, with $a,b$ odd and $c$ even. Proposition~\ref{prop:fermat-difference} excludes this configuration as well. These two cases exhaust the positive primitive equation.

Finally, any non-zero signed solution of $x^3+y^3+z^3=0$ has entries of both signs. Moving the term with the lone sign to the other side and taking absolute values produces a positive solution of $A^3+B^3=C^3$. This proves the signed assertion.
\end{proof}

\bibliographystyle{smfalpha}
\bibliography{fermat_cubic}
\end{document}